\documentclass[journal]{IEEEtran}      

\IEEEoverridecommandlockouts                              

\usepackage{graphicx}
\usepackage{amsfonts,amssymb}
\usepackage{amsmath}
\usepackage{amsthm}
\usepackage{cases}
\usepackage{mathrsfs}
\usepackage{mathabx}
\usepackage{mathtools} 
\usepackage{savesym}

\usepackage{multirow}
\savesymbol{AND}

\usepackage[table]{xcolor}

\usepackage[colorlinks,
linkcolor=blue,
anchorcolor=blue,
citecolor=red
]{hyperref}
\usepackage{cite}
\usepackage{threeparttable}
\usepackage{booktabs,ragged2e}
\usepackage{siunitx}

\usepackage{tikz}
\usetikzlibrary{patterns,decorations.pathmorphing,positioning}

\usepackage{algorithm}
\usepackage{algorithmicx}
\usepackage{algpseudocode}

\newtheorem{theorem}{Theorem}
\newtheorem{lemma}{Lemma}
\newtheorem{remark}{Remark}

\newtheorem{definition}{Definition}

\newtheorem{problem}{Problem}

\newcommand{\calP}{\mathcal{P}}
\newcommand{\calQ}{\mathcal{Q}}
\newcommand{\F}{\mathcal{F}}
\newcommand{\diff}{\mathrm{d}}

\newcommand{\tr}{\mathrm{Tr}}

\begin{document}

	\title{\LARGE \bf
		An Optimal Projection Framework for Structure‑Preserving Model Reduction of Linear Systems}
	
	\author{Xiaodong Cheng, \IEEEmembership{Senior Member, IEEE}
		\thanks{
			X. Cheng is with the Mathematical and Statistical Methods Group (Biometris), Wageningen University \& Research,
			6700 AA Wageningen, The Netherlands.
			{\tt\small xiaodong.cheng@wur.nl}}
	}
	
	\maketitle

	
	\begin{abstract}
		
		This paper presents a model order reduction framework for linear systems where the $\mathcal{H}_2$ optimization is
		combined with the Petrov-Galerkin projection to preserve a wide range of system properties, including stability, bounded realness, sector boundedness, and dissipativity. The model reduction problem is formulated as a nonconvex optimization problem on a generalized Stiefel manifold, aiming to minimize the $\mathcal{H}_2$ norm of the approximation error between the full-order and reduced-order models. An explicit expression for the Euclidean gradient of the objective function is derived, and a Riemannian gradient procedure with an efficient generalized Cayley update is applied to seek a (locally) optimal projection, followed by a theoretical analysis of the algorithm. Finally, the performance of the proposed method is demonstrated by a numerical example.
		
	\end{abstract}
	
	
	\section{Introduction}
	
	Large-scale systems, including electric circuits, structural mechanics, and microelectromechanical systems, are often modeled by high-dimensional dynamical equations which can be computationally demanding for numerical simulations, control design, and optimization.
	This spurs the development of model reduction techniques to generate reduced-order models that capture the pertinent behavior  of the original systems at a much lower computational cost. 
	The reduced-order models can replace their complex counterparts in the control and optimization tasks, resulting in a shorter design cycle time, 
	see \cite{obinata2012model,Gugercin2004Survey,cheng2021review} for an overview.
	Many physical systems possess important structural properties such as stability, passivity, or dissipativity, which determine the system behavior and are crucial for the control and optimization of the systems. It is therefore desirable that reduced‑order models preserve these key properties, ensuring that they remain physically meaningful and suitable for subsequent control and optimization tasks.
	
	Different methods have been developed for reducing large-scale dynamical systems with specific structural properties. The classical balanced truncation approach \cite{moore1981principal} naturally retains system stability and minimality. Various extensions of this approach have been proposed to preserve dissipativity and port-Hamiltonian structures, see, e.g., \cite{polyuga2012effort,guiver2013error,kawano2018structure,salehi2021passivity,XiaodongBT2017}.
	Krylov subspace methods have also been widely used for the reduction of passive systems, see e.g., \cite{antoulas2005passive,sorensen2005passivity,gugercin2012pH,polyuga2010pH,wolf2010passivity,Ionescu2013,breiten2022passivity}. A reduction method for dissipative Hamiltonian systems is presented in \cite{afkham2019structure}, which adopts a symplectic time integrator to conserve the total energy of a Hamiltonian system. However, most existing methods are tailored to preserve a single structural property and often lack guarantees of optimality with respect to any specific performance criterion.
	In contrast, the works such as \cite{Yu1999approximate,vanDooren2008H2,beattie2007krylov,beattie2009trust,gugercin2008H2,sato2017structure,sato2018pHsys,sato2021book,Jiang2019model,moser2020new,schwerdtner2023optimization} formulate model reduction problems in the context of nonconvex $\mathcal{H}_2$ optimization, aiming to minimize the $\mathcal{H}_2$ error between original and reduced-order models.
	In particular, \cite{beattie2007krylov,beattie2009trust,gugercin2008H2} use pole-residue interpolation to derive first-order $\mathcal{H}_2$ optimality conditions, and thereby the rational Krylov algorithm \cite{beattie2007krylov,gugercin2008H2} or the trust region method \cite{beattie2009trust} can be applied to converge to a reduced-order model satisfying first-order optimality conditions with respect to the $\mathcal{H}_2$ error criterion. Nevertheless, structure preservation is not the primary concern of these methods.
	Within the pole-residue framework, \cite{moser2020new} provides an optimal model reduction method that is capable of preserving the port-Hamiltonian structure, while this method is restricted to single-input and single-output systems.
	The other nonconvex optimization methods such as \cite{Yu1999approximate,sato2017structure,sato2018pHsys,sato2021book,Jiang2019model} formulate $\mathcal{H}_2$ optimality using coupled Lyapunov equations and develop gradient-based algorithms on Stiefel or other Riemannian manifolds to obtain structure-preserving reduced models. These methods, however, require the system state matrix or its symmetric part to be negative definite.

    This paper builds upon the Lyapunov-based optimization 	framework without requiring the state matrix or its
	symmetric part to be negative definite. Instead, we
	use a positive-definite structure matrix satisfying
	the corresponding matrix inequality conditions. We adopt an oblique Petrov-Galerkin projector to construct reduced models that minimize the $\mathcal{H}_2$ error while preserving multiple structural properties  of interest, such as stability, bounded realness, sector boundedness, and dissipativity. These properties are encoded via linear matrix inequalities (LMIs), whose solution defines a structure matrix used to construct a projection.  Then, the model reduction problem is formulated as a nonconvex optimization on a compact generalized Stiefel manifold. By using the controllability and observability Gramians of the error system, we derive the explicit expression of the Euclidean gradient of its objective function, identify the corresponding Riemannian gradient under a weighted metric, and propose an efficient generalized Cayley algorithm to compute a locally optimal projection that ensures both low reduction error and the preservation of desirable structural properties.


	The paper is organized as follows. Section~\ref{sec:Preliminaries} introduces the problem setting for structure-preserving model reduction. The main results are presented in Sections~\ref{sec:projection} and~\ref{sec:gradient}, where we formulate the model reduction problem as a nonconvex optimization and develop gradient-based solution methods. A numerical example is provided in Section~\ref{sec:example}, and concluding remarks are given in Section~\ref{sec:conclusion}.

	\textit{Notation.}  Let $\mathbb{R}$ denote the set of real numbers. The squared Frobenius norm of a matrix $A \in \mathbb{R}^{n \times m}$ is denoted by $\|A\|^2_F = \langle A, A \rangle = \tr(A^\top A)$. In a symmetric block matrix, the symbol $\star$ stands for the symmetric counterpart.
	Let $\mathbb{S}^{n \times r}_*$ be
	the set of all full‑column‑rank  $n \times r$ matrices and 
	$\mathrm{St}(n, r)$ be the compact Stiefel manifold: 
	$
	\mathrm{St}(n, r) = \{V \in  \mathbb{S}^{n \times r}_* \mid V^\top V = I_r \}.
	$
	More generally, for a fixed matrix $M=M^\top\succ0$, we define the generalized Stiefel manifold
	\begin{equation}\label{eq:generalized_Stiefel}
		\mathcal M=\mathrm{St}_{M}(n,r)
		:=\{V\in\mathbb{R}^{n\times r}\mid V^\top M V=I_r\},
	\end{equation}
	which is compact and diffeomorphic to $\mathrm{St}(n,r)$ through the transformation $U=M^{1/2}V$.

	\section{Preliminaries \& Problem Setting}
	\label{sec:Preliminaries}

	\subsection{System Properties and Characterizations}
	\label{sec:problem}
	Consider a linear time-invariant system with the following state-space representation.
	\begin{equation}\label{sys:orig}
		\mathbf{\Sigma}: \
		\begin{cases}
			\dot{x}(t) = A x(t) + B u(t), \\
			y(t) = C x(t) + D u(t),
		\end{cases}
	\end{equation}
	where $x(t) \in \mathbb{R}^n$, $u(t) \in \mathbb{R}^p$, and $y(t) \in \mathbb{R}^{q}$ are the vectors of the states, inputs, and outputs of the system $\mathbf{\Sigma}$. $A \in \mathbb{R}^{n \times n}$, $B \in \mathbb{R}^{n \times p}$, $C \in \mathbb{R}^{q \times n}$,  and $D \in \mathbb{R}^{q \times p}$ are constant matrices.
	Assuming $A$ is Hurwitz, i.e., the system $\mathbf{\Sigma}$ is (asymptotically) stable, this paper aims to develop a general framework of model reduction that is applicable to a wide range of systems with specific structural properties. In particular, the properties  of interest in this paper and their corresponding characterizations in terms of Lyapunov inequalities are summarized in Table~\ref{tab:properties}.
	Here we introduce the notion of  $(Q, S, R)$-dissipativity (see e.g., \cite{willems2007dissipative,xia2016passivity}) and then show how passivity, finite-gain $L_2$ stability, and sector boundedness arise as special cases.
	\begin{definition}[$(Q,S,R)$-dissipativity]
			Let the supply rate be defined by
			$
			w(u,y)=y^\top Qy+2y^\top Su+u^\top Ru,
			$
			where $Q=Q^\top$, $R=R^\top$, and $S$ has compatible dimensions. The system $\mathbf{\Sigma}$ in \eqref{sys:orig} is called $(Q,S,R)$-dissipative if there exists a nonnegative storage function $V_s:\mathbb{R}^n\to\mathbb{R}_{\ge 0}$ such that, for every admissible input $u$ and every $\tau\ge 0$,
			\[
			V_s(x(\tau))-V_s(x(0))
			\le
			\int_0^\tau w(u(t),y(t))\,dt.
			\]
		\end{definition}
		
		For the stable LTI system \eqref{sys:orig}, consider the quadratic storage function $ 
		V_s(x)=x^\top Kx$, with $K=K^\top \succ 0$.  
		If \eqref{eq:K_QSR} holds for some $K\succ0$,
		then the system is $(Q,S,R)$-dissipative.
		If the same $K$ also satisfies \eqref{eq:K_stability},
		then the system is asymptotically stable. The notion of $(Q,S,R)$-dissipativity is general and covers several classical system properties of practical interest.
	\begin{enumerate}
		
		\item If $Q = 0$, $S = I$, $R = 0$, the $(Q, S, R)$-dissipativity reduces to the \textit{passivity} (or \textit{positive realness}), i.e. when $x(0)=0$, 
		$
		\int_{0}^{\tau} u(t)^\top y(t) \diff{t} \geq 0
		$
		for all admissible $u(t)$ and all $\tau \geq 0$.
		The inequality \eqref{eq:K_QSR} reduces to \eqref{eq:K_passivity}.
		
		\item If $Q = -\gamma^{-1} I$, $S = 0$, $R = \gamma I$, 
		the $(Q, S, R)$-dissipativity becomes the \textit{finite-gain $L_2$ stability},
		i.e. when $x(0)=0$, 
		$
		\int_{0}^{\tau} y(t)^\top y(t) \diff{t} \leq \gamma^2 \int_{0}^{\tau} u(t)^\top u(t) \diff{t}
		$
		for all admissible $u(t)$ and all $\tau \geq 0$, and \eqref{eq:K_QSR} reduces, by the Schur complement, to \eqref{eq:K_L2gain}. In particular, if $\gamma = 1$, this property is also known as \textit{bounded realness}.
		
		\item If $Q = - I$, $S = \frac{a+b}{2}I$, $R = -abI$, with $a < b$ and $b>0$, then
		the $(Q, S, R)$-dissipativity implies the input-output map lies in the conic sector $[a, b]$, i.e. when $x(0)=0$, 
		$
		\left(1 + \frac{a}{b}\right) \int_{0}^{\tau} y(t)^\top u(t) \diff{t}  \geq \frac{1}{b}\int_{0}^{\tau} y(t)^\top y(t) \diff{t} +  a \int_{0}^{\tau} u(t)^\top u(t) \diff{t}
		$
		holds for all admissible $u(t)$ and all $\tau \geq 0$. The LMI \eqref{eq:K_QSR} reduces to \eqref{eq:K_sector}.
		
	\end{enumerate}

		
		

	\begin{table}[t]
		\centering
		\begin{threeparttable}
			\caption{Conditions for Key System Properties\tnote{*}}\label{tab:properties}
			\begin{tabular}{m{0.93\linewidth}}
				\toprule
				\rowcolor{gray!10}    Asymptotic stability           \\
				\begin{minipage}{\linewidth}
					\centering
					\begin{equation} \label{eq:K_stability}
						A^\top K + K A \prec 0
					\end{equation}
					\vfil
				\end{minipage}                          \\
				\rowcolor{gray!10}  Passivity (or positive realness) \\
				\begin{minipage}{\linewidth}
					\centering
					\begin{equation} \label{eq:K_passivity}
						\begin{bmatrix}
							A^\top K + K A & KB - C^\top   \\
							\star          & {-D^\top - D}
						\end{bmatrix} \preceq  0,
					\end{equation}
					\vfil
				\end{minipage}                          \\
				\rowcolor{gray!10}  Finite-gain $L_2$ stability      \\
				\begin{minipage}{\linewidth}
					\centering
					\begin{equation} \label{eq:K_L2gain}
						\begin{bmatrix}
							A^\top K + K A & KB         \\
							\star          & - \gamma I
						\end{bmatrix} + \frac{1}{\gamma}
						\begin{bmatrix}  C^\top \\ 	D^\top 	\end{bmatrix}
						\begin{bmatrix} C & D  \end{bmatrix}
						\preceq 0
					\end{equation}
					\vfil
				\end{minipage}                          \\
				\rowcolor{gray!10}  Conic sector boundedness         \\
				\begin{minipage}{\linewidth}
					\centering
					\begin{equation} \label{eq:K_sector}
						\begin{bmatrix}
							A^\top K + K A + C^\top C & KB - \frac{a+b}{2}C^\top + C^\top D              \\
							\star                     & {D^\top D} - \frac{(a+b)}{2} (D^\top + D) + ab I
						\end{bmatrix} \preceq 0
					\end{equation}
					\vfil
				\end{minipage}                          \\
				\rowcolor{gray!10}  $(Q, S, R)$-dissipativity        \\
				\begin{minipage}{\linewidth}
					\centering
					\begin{equation} \label{eq:K_QSR}
						\begin{bmatrix}
							A^\top K + K A - C^\top Q C & KB - C^\top S - C^\top Q D             \\
							\star                       & - D^\top Q D - D^\top S - S^\top D - R
						\end{bmatrix} \preceq 0
					\end{equation}
				\end{minipage}                          \\
				\bottomrule
			\end{tabular}
			\begin{tablenotes}
				\item[*] A sufficient condition for each property is a matrix $K\succ0$ satisfying the corresponding inequality. Furthermore, for passivity and conic sector boundedness,
				$p=q$ is required.
			\end{tablenotes}
		\end{threeparttable}
	\end{table}
	The first matrix inequality in
	Table~\ref{tab:properties} provides a necessary
	and sufficient condition for asymptotic stability,
	while the other four inequalities provide sufficient
	conditions for the corresponding dissipativity
	properties. Each of these LMIs guarantees the
	existence of a quadratic storage function
	$V_s(x)=x^\top Kx$ satisfying the associated
	dissipation inequality. These sufficient conditions
	do not require the system $(A,B,C,D)$ to be minimal.
	However, their necessity generally depends on
	additional assumptions, such as controllability
	or minimality, and is not considered here.
	
	\begin{remark}
		The notion of $(Q,S,R)$-dissipativity provides a unified description of several important input-output properties through suitable choices of $(Q,S,R)$, including passivity, finite-gain $L_2$ stability, and conic sector boundedness. Furthermore, the corresponding certificate matrix $K$ is used directly to define the structure matrix $X$ in the projection framework of Section~\ref{sec:projection}. Therefore, dissipativity is not only a property to be verified, but also the mechanism that determines the admissible structure-preserving projection.
	\end{remark}
	
	\subsection{Problem Formulation}
	
	Consider the model reduction problem for a linear system in \eqref{sys:orig} with specific desired structural properties in Table~\ref{tab:properties}, our objective is to construct a reduced-order model that retains these properties while minimizing the reduction error. Specifically, the research problem is formulated as follows.
	\begin{problem}
		Consider a full-order model $\mathbf{\Sigma}$ in \eqref{sys:orig} that is asymptotically stable (i.e. $A$ is Hurwitz) and satisfies additional properties in Table~\ref{tab:properties}, find a reduced-order system of dimension $r$ ($r < n$):
		\begin{equation}\label{sys:red0}
			\mathbf{\hat{\Sigma}}: \
			\begin{cases}
				\dot{\hat{x}}(t)  = \hat{A} \hat{x}(t) +  \hat{B} u(t), \\
				\hat{y}(t) = \hat{C}   \hat{x}(t) + \hat{D} u(t),
			\end{cases}
		\end{equation}
		where $\hat{A} \in \mathbb{R}^{r \times r}$, $\hat{B}\in \mathbb{R}^{r \times p}$, $\hat{C} \in \mathbb{R}^{q \times r}$, and $\hat{D} \in \mathbb{R}^{q \times p}$ are reduced matrices such that:
		\begin{enumerate}
			\item Asymptotic stability and the specified structural properties of the original system are preserved.
			
			\item The reduced-order model $\hat{\Sigma}$ provides a (locally) optimal approximation of the original model $\Sigma$ with respect to the  $\mathcal{H}_2$ norm of the approximation error.
		\end{enumerate}
	\end{problem}
	
	
	
	To address the above problem,  Section~\ref{sec:projection} introduces an oblique projection method, which enables the construction of a family of reduced-order models that retain the desired properties. Then, in Section~\ref{sec:gradient}, the model reduction problem is reformulated and solved as a nonconvex optimization problem.
	
	\section{Projection-Based Model Reduction With Structure Preservation}
	\label{sec:projection}
	
	Consider a full-order model in \eqref{sys:orig} with any structural property in Table~\ref{tab:properties}. To preserve this property in the reduced-order model \eqref{sys:red0}, the following oblique projection $\Pi$ is adopted:
	\begin{align}\label{eq:pinv}
		\Pi  = V V^\dagger \quad \text{with} \ V^\dagger := (V^\top X V)^{-1} V^\top X,
	\end{align}
	where $V \in \mathbb{R}^{n \times r}$ has full column
	rank, and $X \in \mathbb{R}^{n \times n}$ is a symmetric
	positive definite matrix, referred to as the
	\textit{structure matrix}. The matrix $V$ defines the
	trial space, while $X V$ spans the test space. The matrix $V^\dagger$ is an $X$-weighted left inverse
	of $V$ and satisfies $V^\dagger V = I_r$.
	It is well-defined since $V^\top X V$ is positive
	definite. Consequently, $\Pi$ is an $X$-orthogonal
	projection onto the range of $V$, which is generally
	oblique in the Euclidean inner product.
	The reduced-order model is then constructed as follows.
	\begin{align}\label{sys:red}
		\mathbf{\hat{\Sigma}}: \
		\begin{cases}
			\dot{\hat{x}}(t)  = \underbrace{\mathclap{V^\dagger A V}}_{\hat{A}} \hat{x}(t) + \underbrace{V^\dagger B}_{\hat{B}} u(t), \\
			\hat{y}(t) = \underbrace{C V}_{\hat C} \hat{x}(t) + D u(t),
		\end{cases}
	\end{align}
	where $\hat{x} \in \mathbb{R}^{r}$ is the reduced state vector. The following result shows that this particular construction of a reduced-order model preserves the specified structural property under the conditions below.
	\begin{theorem}[Structure Preservation]
		\label{thm:preservation}
		Consider the original system $\mathbf{\Sigma}$
		in \eqref{sys:orig} and its reduced-order model
		$\mathbf{\hat{\Sigma}}$ in \eqref{sys:red},
		constructed using the projection \eqref{eq:pinv}
		with $X=K \succ0$.
		Then, for any $V\in\mathbb{S}^{n\times r}_*$,
		the following statements hold:
		
		\begin{enumerate}		
			\item If $K$ satisfies \eqref{eq:K_stability}, then $\mathbf{\hat{\Sigma}}$ is asymptotically stable.
				
			\item If $K$ satisfies \eqref{eq:K_stability} and \eqref{eq:K_passivity}, then $\mathbf{\hat{\Sigma}}$ is
			asymptotically stable and passive.

			\item If $K$ satisfies \eqref{eq:K_stability} and \eqref{eq:K_L2gain} for some $\gamma>0$, then
			$\mathbf{\hat{\Sigma}}$ is asymptotically
			stable and has finite $L_2$ gain bounded
			by the same $\gamma$.
			
			\item If $K$ satisfies \eqref{eq:K_stability} and \eqref{eq:K_sector}, then $\mathbf{\hat{\Sigma}}$ is
			asymptotically stable and lies in the same
			conic sector $[a,b]$.
			
			\item If $K$ satisfies both
			\eqref{eq:K_stability} and \eqref{eq:K_QSR},
			then $\mathbf{\hat{\Sigma}}$ is
			asymptotically stable and $(Q,S,R)$-dissipative
			with the same supply-rate matrices $Q$, $S$,
			and $R$.

		\end{enumerate}

	\end{theorem}
	The proof is provided in Appendix~\ref{ap:thm:preservation}. Different structural properties can be preserved in the reduced-order model \eqref{sys:red} by altering the structure matrix $X$.
	For the optimization considered below, we further
	assume that the structure matrix $X$ satisfies
	\eqref{eq:K_stability}. Then the reduced-order model
	is asymptotically stable for every full-column-rank
	projection basis $V$, and the $\mathcal H_2$ objective
	is well-defined on the entire generalized Stiefel
	manifold $\mathcal M$.
	%
	
	\begin{remark}
		The proposed framework provides a method for the passivity-preserving model reduction for port-Hamiltonian systems, which have been extensively studied in the literature, see e.g., \cite{antoulas2005passive,gugercin2012pH,polyuga2010pH,wolf2010passivity,Ionescu2013,hauschild2019pH,salehi2021passivity,sato2018pHsys,moser2020new}.
		In our approach, the structure matrix $X$ in \eqref{eq:pinv} can be chosen as the energy matrix $K$ of the system i.e. $\frac{1}{2}x(t)^\top K x(t)$ represents the total energy (Hamiltonian) of the system. Consequently, the reduced-order model has the reduced energy matrix $V^\top K V$.
	\end{remark}



	The reduced-order model in \eqref{sys:red} guarantees the preservation of a desirable structural property by selecting a specific structure matrix $X$. Then, we can minimize the approximation error between the original and reduced models by tuning $V$. To this end, we define the error system as
	\begin{align} \label{sys:err}
		\begin{cases}
			\dot{z}(t) = {A}_e z(t) + {B}_eu(t ) \\
			y_e (t) = {C}_e z(t)
		\end{cases}
	\end{align}
	where $z(t) = [x(t)^\top \ \hat{x}(t)^\top]^\top \in \mathbb{R}^{n+r}$, and
	\begin{align*}
		{A}_e = \begin{bmatrix}
			A & 0 \\ 0 & V^\dagger A  V
		\end{bmatrix}, \
		{B}_e = \begin{bmatrix}
			B \\  V^\dagger B
		\end{bmatrix}, \
		{C}_e = \begin{bmatrix}
			C & - CV
		\end{bmatrix}.
	\end{align*}
	Then, because both the original and reduced-order systems are asymptotically stable, the error system is asymptotically stable, and hence  the squared $\mathcal{H}_2$-norm of the approximation error can be computed as
	$
	\| G_e(s) \|_{\mathcal{H}_2}^2 =
	\tr(C_e \calP C_e^\top) = \tr(B_e^\top \calQ B_e),
	$
	where $\calP$ and $\calQ$ are the controllability and observability Gramians of the error system \eqref{sys:err}. Algebraically, $\calP$ and $\calQ$ are solved as the unique solutions of the following Lyapunov equations:
	\begin{subequations}\label{eq_Gramians_err}
		\begin{align}
			\label{W_Lyap}
			A_e \calP+\calP A_e^\top+ B_e B_e^\top & =0, \\
			\label{M_Lyap}
			A_e^\top\calQ+\calQ A_e + C_e^\top C_e & =0.
		\end{align}
	\end{subequations}
	Following standard techniques in e.g., \cite{vanDooren2008H2,beattie2009trust,sato2017structure,Jiang2019model} to analyze the error system, the Gramians $\calP$ and $\calQ$ are partitioned according to the block structure of the error system as:
	\begin{equation} \label{eq:blockMW}
		\calP = \begin{bmatrix}
			\calP_{11} & \calP_{12} \\ \calP_{12}^\top & \calP_{22}
		\end{bmatrix}, \quad
		\calQ = \begin{bmatrix}
			\calQ_{11} & \calQ_{12} \\ \calQ_{12}^\top & \calQ_{22}
		\end{bmatrix},
	\end{equation}
	where $\calP_{11}$ and $\calQ_{11}$ are Gramians of the high-order system \eqref{sys:orig}, and $\calP_{22}$ and $\calQ_{22}$ are the Gramians of the reduced-order model \eqref{sys:red}. Here, the submatrices $\calP_{22}$, $\calQ_{22}$, $\calP_{12}$, and $\calQ_{12}$ are characterized by the following Sylvester/Lyapunov equations
	\begin{subequations} \label{eq:Sylv}
		\begin{align}
			\hat{A}	\calP_{22} + \calP_{22}\hat{A}^\top + \hat{B} \hat{B}^\top = 0,
			\\
			\hat{A}^\top	\calQ_{22} + \calQ_{22}\hat{A}  + \hat{C}^\top \hat{C}  = 0,
			\\
			{A}	\calP_{12} + \calP_{12}  \hat{A}^\top +  {B} \hat{B}^\top = 0,
			\label{eq:Sylv_c}
			\\
			{A}^\top	\calQ_{12} + \calQ_{12}\hat{A}  -  {C}^\top \hat{C}  = 0,
			\label{eq:Sylv_d}
		\end{align}
	\end{subequations}
	where $\hat{A}$, $\hat{B}$, and $\hat{C}$ are computed as in \eqref{sys:red}.

	Observe that given a structure matrix $X$, the reduced-order model \eqref{sys:red} is parametrized by the matrix $V$.
	This parametrization is basis-invariant. For every nonsingular matrix $T\in\mathbb{R}^{r\times r}$,
	\begin{align}\label{eq:basis_invariance}
		(VT)^\dagger&=T^{-1}V^\dagger, &
		\hat A(VT)&=T^{-1}\hat A(V)T, \nonumber\\
		\hat B(VT)&=T^{-1}\hat B(V), &
		\hat C(VT)&=\hat C(V)T.
	\end{align}
	Thus, $V$ and $VT$ produce reduced-order models with the same transfer matrix and satisfy $J(VT)=J(V)$. Consequently, the problem is intrinsically defined by the trial space, i.e., the range of $V$. For any $V\in\mathbb{S}^{n\times r}_*$, the matrix
	$
	V(V^\top M V)^{-1/2}
	$
	belongs to $\mathcal M$ and spans the same subspace. Therefore, restricting the search to the generalized Stiefel manifold does not exclude any reduced transfer function represented by a full-column-rank basis.

We formulate an optimization problem in terms of minimizing the squared $\mathcal{H}_2$ norm of the error system $G_e(s)$ on $\mathcal M$:
	\begin{align}
		\label{eq:optimization1}
		\min_{V \in \mathcal M} \;  J(V)  \; & :=  \tr \left( B_e(V)^\top  \calQ(V)  B_e(V) \right)     \\
		\;\; \text{s.t.:}  \;\;                              & A_e^\top \calQ + \calQ A_e + C_e^\top C_e =0.
		\nonumber
	\end{align}
	This optimization problem is {nonconvex} with $\calQ =\calQ(V)$ an implicit function of $V$ due to the unique solution of the Lyapunov equation \eqref{M_Lyap}. Nevertheless, its feasible set is compact. The matrix $M$ in \eqref{eq:generalized_Stiefel} affects the representation and the geometry of the optimization algorithm, but not the family of attainable reduced transfer functions. A particularly natural choice is $M=X$. In the following sections, we will investigate how to solve this optimization problem.

	\section{Gradient-Based Methods for Optimal Projection}
	\label{sec:gradient}
	This section solves the model reduction problem as a nonconvex optimization problem. To this end, an analytical expression for the gradient of the objective function $J(V)$ with respect to the projection basis $V$ is derived, and the properties of the gradient flow are also analyzed. Then, a Riemannian gradient algorithm is provided to find a (locally) optimal projection on $\mathcal{M}$.
	
	\subsection{Gradient Analysis}
	
	To solve the optimization problem \eqref{eq:optimization1}, we analyze the key properties of the objective function $J(V)$ such as smoothness and differentiability, based on which we can apply gradient-based methods for solving \eqref{eq:optimization1} over the manifold $\mathcal M$. The next lemma shows that $J$ is a differentiable function and derives an explicit expression for its Euclidean gradient.
	\begin{lemma}[Gradient on Matrix Manifold]
		\label{lem:gradJ}
		The objective function $J(V)$ in  \eqref{eq:optimization1} is differentiable, and the gradient has the following expression:
		\begin{align}  \label{eq:nablaJ}
			\nabla J(V)
			= & 2 X (I - V V^\dagger)  \left[ \F_A (V) + \F_B (V)^\top \right]   (V^\top X V)^{-1}
			\nonumber                                                                              \\
			& - 2(V^\dagger)^\top  \left[ \F_A (V)^\top + \F_B (V) \right]  (V^\dagger)^\top
			\nonumber
			\\& 	 +  2 A^\top (V^\dagger)^\top \Sigma_{22}^\top   + 2 \F_C (V),
		\end{align}
		where $\Sigma_{22} := \calP_{12}^\top\calQ_{12} +  \calP_{22}\calQ_{22}$, and
		\begin{subequations}\label{eq:FAFBFC}
			\begin{align}
				\F_A (V)  & : = A V    \Sigma_{22}, \label{eq:F_A}                                   \\
				\F_B (V)  & : = (\calQ_{12}^\top    +  \calQ_{22} V^\dagger) BB^\top, \label{eq:F_B} \\
				\F_C (V)  & :=  C^\top C (V \calP_{22} -  \calP_{12}).
			\end{align}
		\end{subequations}
	\end{lemma}
	%
	The detailed derivation of the expression \eqref{eq:nablaJ} can be found in Appendix~\ref{ap:lem:gradJ}.
		Note that different from methods in e.g., \cite{vanDooren2008H2,sato2018pHsys}, which compute gradients with respect to the reduced matrices $\hat{A}$, $\hat{B}$, and $\hat{C}$, this paper focuses on optimizing the reduced basis matrix $V$, which is the only optimization variable in the optimization problem \eqref{eq:optimization1}.

	With the formula for the Euclidean gradient $\nabla J$ in \eqref{eq:nablaJ}, we equip $\mathcal M$ with the metric
	\begin{equation}\label{eq:M_metric}
		g_V(Z_1,Z_2):=\tr(Z_1^\top M Z_2),
	\end{equation}
	whose tangent space at $V$ is given by
	\begin{equation}\label{eq:M_tangent}
		T_V\mathcal M
		=\{Z\in\mathbb{R}^{n\times r}\mid V^\top M Z+Z^\top M V=0\}.
	\end{equation}
	The following result connects the Euclidean gradient in Lemma~\ref{lem:gradJ} with the Riemannian geometry of $\mathcal M$ and characterizes the resulting gradient flow.

	\begin{theorem} \label{thm:gradflow}
		The following statements hold.
		\begin{enumerate}
			\item For every $V\in\mathcal M$, the Euclidean gradient satisfies
			\begin{equation}\label{eq:vertical_orthogonality}
				V^\top\nabla J(V)=0.
			\end{equation}
			Moreover, under the metric \eqref{eq:M_metric}, the Riemannian gradient is
			\begin{equation}\label{eq:Riemannian_gradient}
				\operatorname{grad}J(V)=M^{-1}\nabla J(V).
			\end{equation}

			\item For every $V\in\mathcal M$, define
			\[
				W_M(V):=M^{-1}\nabla J(V)V^\top M-V\nabla J(V)^\top.
			\]
			Then
			\begin{equation} \label{eq:JV0}
				\nabla J(V)=0\quad\Longleftrightarrow\quad W_M(V)=0.
			\end{equation}

			\item Let $V(t)$ satisfy the Riemannian gradient flow
			\begin{equation}\label{eq:gradflow}
				\frac{\diff}{\diff t}V(t)=-M^{-1}\nabla J(V(t)),
			\end{equation}
			with $V(0)=V_0\in\mathcal M$. Then
			\[
				V(t)^\top M\frac{\diff}{\diff t}V(t)=0,
				\qquad V(t)^\top M V(t)=I_r,
			\]
			and hence $V(t)\in\mathcal M$ for all $t\geq0$.

			\item Along the flow \eqref{eq:gradflow}, the objective function is nonincreasing, with
			\begin{align} \label{eq:nonincreasing}
				\frac{\diff}{\diff t}J(V(t))
				=-\tr\!\left(\nabla J(V(t))^\top M^{-1}\nabla J(V(t))\right)\leq0.
			\end{align}
		\end{enumerate}
	\end{theorem}
	
	The proof is presented in Appendix~\ref{ap:thm:gradflow}. The first two statements give the static gradient properties, while the last two statements show that the gradient flow remains on the generalized Stiefel manifold and decreases the objective function. Since $J$ is
	smooth and the manifold is compact, standard
	gradient-flow arguments imply that every
	accumulation point of the gradient flow is a
	critical point of $J$. Owing to the basis
	invariance \eqref{eq:basis_invariance}, these
	critical points represent critical subspaces,
	or equivalently, critical points on the
	quotient $\mathcal M/O(r)$.
	The characterization \eqref{eq:JV0} will be
	used in the generalized Cayley update developed
	below.
	To obtain a practical iterative method,
	we next consider a discrete implementation
	of the gradient flow using the explicit
	gradient \eqref{eq:nablaJ}.  
	
	\subsection{Gradient-Based Algorithms}

	Let $X$ be chosen based on the structural property of the original system \eqref{sys:orig} to be preserved, and let $M=M^\top\succ0$ specify the generalized Stiefel manifold and metric. Choose $V_0\in\mathcal M$ as an initial solution. A direct discretization of the gradient flow is
	\begin{align} \label{eq:simple_iter}
		V_{k+1}=V_k-\alpha_kM^{-1}\nabla J(V_k),
	\end{align}
	with stepsize $\alpha_k>0$. Since $V_k^\top\nabla J(V_k)=0$ due to the
	basis invariance of $J$, we have
	\[
		V_{k+1}^\top M V_{k+1}
		=V_k^\top M V_k +\alpha_k^2\nabla J(V_k)^\top M^{-1}\nabla J(V_k)\succ0.
	\]
	Thus, every iterate remains full column rank. However, this update generally does not preserve
	the constraint $V^\top M V=I_r$. Consequently,
	the iterates may leave the compact manifold
	$\mathcal M$ and evolve in the noncompact set
	$\mathbb{S}_*^{n\times r}$, making the convergence analysis more challenging.

	To obtain a feasible iterative scheme with a convergence guarantee, introduce the following curvilinear-search update
	\begin{equation} \label{eq:Vt_new}
	V_{k+1}
		=V_k-\tau_k\left(I+\frac{\tau_k}{2}V_k\nabla J(V_k)^\top\right)
		M^{-1}\nabla J(V_k)\mathcal U_k^{-1},
	\end{equation}
	where
	$
		\mathcal U_k
		:=I_r+\frac{\tau_k^2}{4}\nabla J(V_k)^\top M^{-1}\nabla J(V_k),
	$
	and $\tau_k>0$ is the step size. For the iterate $V_k$, define the   search curve as
	\begin{equation}\label{eq:Vtau}
	\mathcal V_k(\tau)
		=V_k-\tau\left(I+\frac{\tau}{2}V_k\nabla J(V_k)^\top\right)
		M^{-1}\nabla J(V_k)\mathcal U(\tau)^{-1},
	\end{equation}
	where
	$
		\mathcal U(\tau)
		=I_r+\frac{\tau^2}{4}\nabla J(V_k)^\top M^{-1}\nabla J(V_k).
	$
	Then \eqref{eq:Vt_new} is obtained by setting $V_{k+1}=\mathcal V_k(\tau_k)$. The next theorem
	gives the main properties of the curvilinear-search scheme.

	\begin{theorem}
		\label{thm:alg2}
		Let $M=M^\top\succ0$ and $V_0\in\mathcal M$. Let $\{V_k\}$ be generated by \eqref{eq:Vt_new}, where $\tau_k$ satisfies the Armijo condition \eqref{eq:Armijo2}. Then, the sequence $\{V_k\}$ has the following properties.
		\begin{enumerate}
			\item $V_k\in\mathcal M$, i.e., $V_k^\top M V_k=I_r$, for each iteration $k$.
			\item The objective function is nonincreasing along the iterations, with
			\begin{align} \label{eq:convergence}
				J(V_{k+1})
				&\leq J(V_k)-\sigma\tau_k
				\tr\!\left(\nabla J(V_k)^\top M^{-1}\nabla J(V_k)\right).
			\end{align}
			\item The sequence $\{J(V_k)\}$ is monotonically nonincreasing and convergent. Moreover, every accumulation point $V_\star$ of $\{V_k\}$ is a critical point of $J$ on $\mathcal M$, which is equivalent to $\nabla J(V_\star)=0$.
		\end{enumerate}
	\end{theorem}
	The detailed proof is given in Appendix~\ref{ap:thm:alg2}. In particular, \eqref{eq:Vt_new} preserves the generalized Stiefel manifold. When $M=I_n$, it reduces to the update on the ordinary Stiefel manifold.

	The stepsize $\tau_k$ is selected according to the Armijo backtracking rule \cite{sun2006optimization,Absil2009optimization}, namely
	\begin{align}
		J(\mathcal V_k(\beta^{m_k}\bar\tau))
		\leq J(\mathcal V_k(0))
		+\sigma\beta^{m_k}\bar\tau\,
		\phi_k'(0),
		\label{eq:Armijo2}
	\end{align}
	where $\beta,\sigma\in(0,1)$, $\bar\tau>0$, $m_k$ is the smallest nonnegative integer satisfying \eqref{eq:Armijo2}, and
	\begin{equation*} 
		\phi_k'(0)
		:={\left.\frac{\diff}{\diff\tau}J(\mathcal V_k(\tau))\right|}_{\tau=0}
		=-\tr\!\left(\nabla J(V_k)^\top M^{-1}\nabla J(V_k)\right).
	\end{equation*}
	The stepsize is then $\tau_k=\beta^{m_k}\bar\tau$. In practice, each candidate stepsize requires an evaluation of the objective along the curve and therefore an additional solution of the Sylvester/Lyapunov equations in \eqref{eq:Sylv}. A heuristic adaptive multiplicative strategy can also be used, but the convergence statement in Theorem~\ref{thm:alg2} is associated with the Armijo rule.

\begin{algorithm}[t]
	\small
	\caption{Generalized Stiefel Curvilinear Search with Armijo Backtracking}
	\label{alg2}
	\begin{algorithmic}[1]
		\State Choose $X\succ0$ satisfying the relevant structural LMI and \eqref{eq:K_stability}, let $\widetilde V_0\in\mathbb{S}^{n\times r}_*$, and choose $M=M^\top\succ0$ 
		\State Set $V_0=\widetilde V_0(\widetilde V_0^\top M\widetilde V_0)^{-1/2}$. Choose $\bar\tau>0$, $\beta,\sigma\in(0,1)$, maximum number of iterations $N_\mathrm{max}$, and tolerance $\epsilon>0$.
		\State $k\gets0$.
		\Repeat
		\State Solve the Sylvester equations in \eqref{eq:Sylv} for $V_k$.
		\State Compute $J(V_k)$ and $\nabla J(V_k)$ using \eqref{eq:nablaJ}.
		\State Solve $M Z_k=\nabla J(V_k)$.
		\If{$\tr(\nabla J(V_k)^\top Z_k)\leq\epsilon$}
		 \textbf{break}
		\EndIf
		\State Set $\tau\gets\bar\tau$.
		\Repeat
		\State $\mathcal U\gets I_r+\frac{\tau^2}{4}\nabla J(V_k)^\top Z_k$.
		\State $\widetilde V\gets V_k-\tau\left(I+\frac{\tau}{2}V_k\nabla J(V_k)^\top\right)Z_k\mathcal U^{-1}$.
		\State Solve \eqref{eq:Sylv} for $\widetilde V$ and evaluate $J(\widetilde V)$.
		\If{$J(\widetilde V)\leq J(V_k)-\sigma\tau\tr(\nabla J(V_k)^\top Z_k)$}
		\State Accept the step: $V_{k+1}\gets\widetilde V$; \textbf{break}
		\Else
		\State Reject the step and set $\tau\gets\beta\tau$.
		\EndIf
		\Until{the step is accepted}
		\State $k\gets k+1$.
		\Until{$k\geq N_\mathrm{max}$}
	\end{algorithmic}
\end{algorithm}

	\textit{Initialization}. In principle, any full-column-rank matrix $\widetilde V_0\in\mathbb{R}^{n\times r}$ can be used to initialize Algorithm~\ref{alg2}. This includes random initializations as well as reduced bases obtained from balanced truncation, Krylov subspace, and clustering-based methods \cite{cheng2021review}. The normalization $V_0=\widetilde V_0(\widetilde V_0^\top M\widetilde V_0)^{-1/2}$ makes the initial point feasible without changing its range. Since the objective function is basis-invariant, $V_0$ and $\widetilde V_0$ represent the same initial trial subspace. Therefore, using a normalized basis from an existing reduction method as $V_0$ provides a natural
	and effective initialization strategy. In this way, the proposed optimization approach can be regarded as a post-processing
	refinement step that further improves the approximation quality of
	reduced-order models obtained by classical methods.

\begin{remark}
	The update \eqref{eq:Vt_new} is essentially a low-dimensional realization of a generalized Cayley-transform step using $V^\top\nabla J(V)=0$. Its practical merit is that it avoids the inversion of an $n\times n$ Cayley matrix and instead requires only the inverse of the $r\times r$ matrix $\mathcal U_k$. The additional operation introduced by the generalized metric is the solution of $MZ_k=\nabla J(V_k)$. For a structured or sparse $M$, it can be implemented by a factorization or an iterative linear solver.
\end{remark}

\begin{remark} 
	\label{rem:Riemannian}
	The optimization problem is intrinsically subspace-based by \eqref{eq:basis_invariance}. Under the metric \eqref{eq:M_metric}, the Riemannian gradient is $M^{-1}\nabla J(V)$, because \eqref{eq:vertical_orthogonality} makes the ambient weighted gradient tangent to $\mathcal M$. Hence, the explicit Euclidean gradient \eqref{eq:nablaJ} is directly compatible with generalized Stiefel optimization after solving a linear system with $M$, and \eqref{eq:Vt_new} is a retraction-type update on this manifold \cite{Absil2009optimization}.
\end{remark}

\section{Numerical Example}
\label{sec:example}
	The proposed optimal projection-based model reduction method (abbreviated as OPMR) is illustrated on a benchmark port-Hamiltonian RLC ladder network adapted from \cite{polyuga2010pH} for passivity-preserving model reduction.  To assess the effectiveness of the method, we consider a network with $n=500$ states and $p = q=6$ inputs/outputs, and the ladder consists of $250$ LC pairs, where all capacitors, inductors and resistive parameters are chosen as $
	C_i = 1, L_i = 1$, $
	R_i = 0.1, $ $i=1,\ldots,250$, respectively, with a terminal load resistance
	$
	R_{\mathrm{load}} = 0.5.$
	The resulting system is written in port-Hamiltonian form, and the structure matrix in \eqref{eq:pinv} is chosen as the Hamiltonian matrix of the system. This choice satisfies the passivity inequality in \eqref{eq:K_passivity}, and therefore ensures that the reduced-order model preserves passivity by Theorem~\ref{thm:preservation}.
	For the generalized Stiefel manifold, we set $M=I_n$ and use the normalized representative $V(V^\top M V)^{-1/2}$.

	The proposed method is compared with two projection-based reduction methods for port-Hamiltonian systems: balanced truncation for port-Hamiltonian systems (BT-PH) \cite{kawano2018structure} and the iterative rational Krylov algorithm for port-Hamiltonian systems (IRKA-PH) \cite{gugercin2012pH}. 
	For OPMR, we compare two stepsize selection
	strategies for the generalized Cayley update
	\eqref{eq:Vt_new}: a heuristic multiplicative
	stepsize adjustment and the Armijo backtracking
	rule \eqref{eq:Armijo2}.
	 In all experiments, the maximum number of OPMR iterations is set to $200$. For each reduced order, the same IRKA-generated initial subspace is used to initialize both OPMR variants.  
	The simulations are performed on an Intel Core i7-7500U (2.70GHz) CPU. 
	Fig.~\ref{fig:errcomp} reports the relative $\mathcal{H}_2$ errors together with the average computation time per OPMR iteration when different reduced orders $r$ are applied. For all tested $r$, both OPMR variants achieve smaller approximation errors than BT-PH and IRKA-PH. The improvement is particularly clear for moderate reduced orders. For example, when $r=10$, the relative $\mathcal{H}_2$ error of OPMR with the adaptive stepsize is $54.15\%$ lower than that of BT-PH and $53.81\%$ lower than that of IRKA-PH. When $r = 12$, OPMR reduces the relative $\mathcal{H}_2$ errors of BT-PH and IRKA-PH by 41.46\% and 
	40.58\%, respectively.
	Furthermore, the Cayley update with Armijo backtracking achieves slightly lower approximation errors than the direct gradient update with the heuristic adaptive stepsize in the tested cases.
	In terms of computational load, the timing bars in Fig.~\ref{fig:errcomp} show that the Armijo rule requires a higher average time per iteration as it involves repeated objective evaluations.
	
	To examine the sensitivity to initialization of the nonconvex objective $J(V)$, Fig.~\ref{fig:converg} presents the convergence curves for different initializations at $r = 8$.  The left panel shows the results for $10$ random full-rank initializations, and the right panel compares three structured initializations, namely clustering-based \cite{cheng2021review}, IRKA-based \cite{gugercin2012pH}, and BT-based  \cite{kawano2018structure} initializations. Although different initializations lead to different convergence trajectories, final objective values are rather close. In particular, the random initializations show a large spread in the initial errors but rapidly converge after about 100 iterations. In contrast, the structured initializations start from much smaller initial errors and converge faster to slightly better final values. Among them, IRKA- and BT-based initializations show similar convergence behavior and outperform the clustering-based approach by a small margin.

	\begin{figure}[t]
		\centering
		\includegraphics[width=0.45\textwidth]{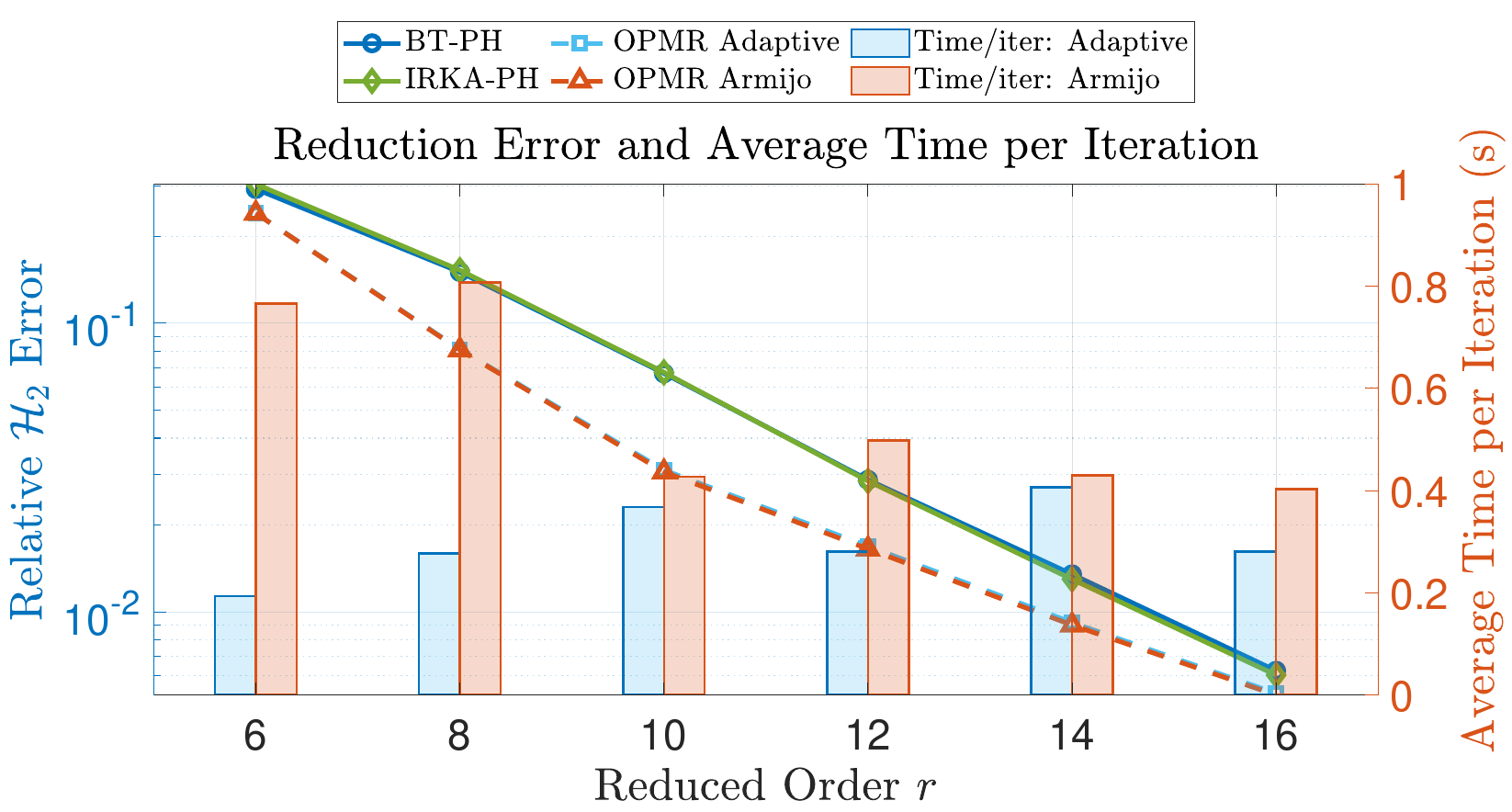}
		\caption{Relative $\mathcal{H}_2$ errors and average computation time per iteration as the reduced order $r$ varies.}
		\label{fig:errcomp}
	\end{figure}
	\begin{figure}[t]
		\centering
		\includegraphics[width=0.48\textwidth]{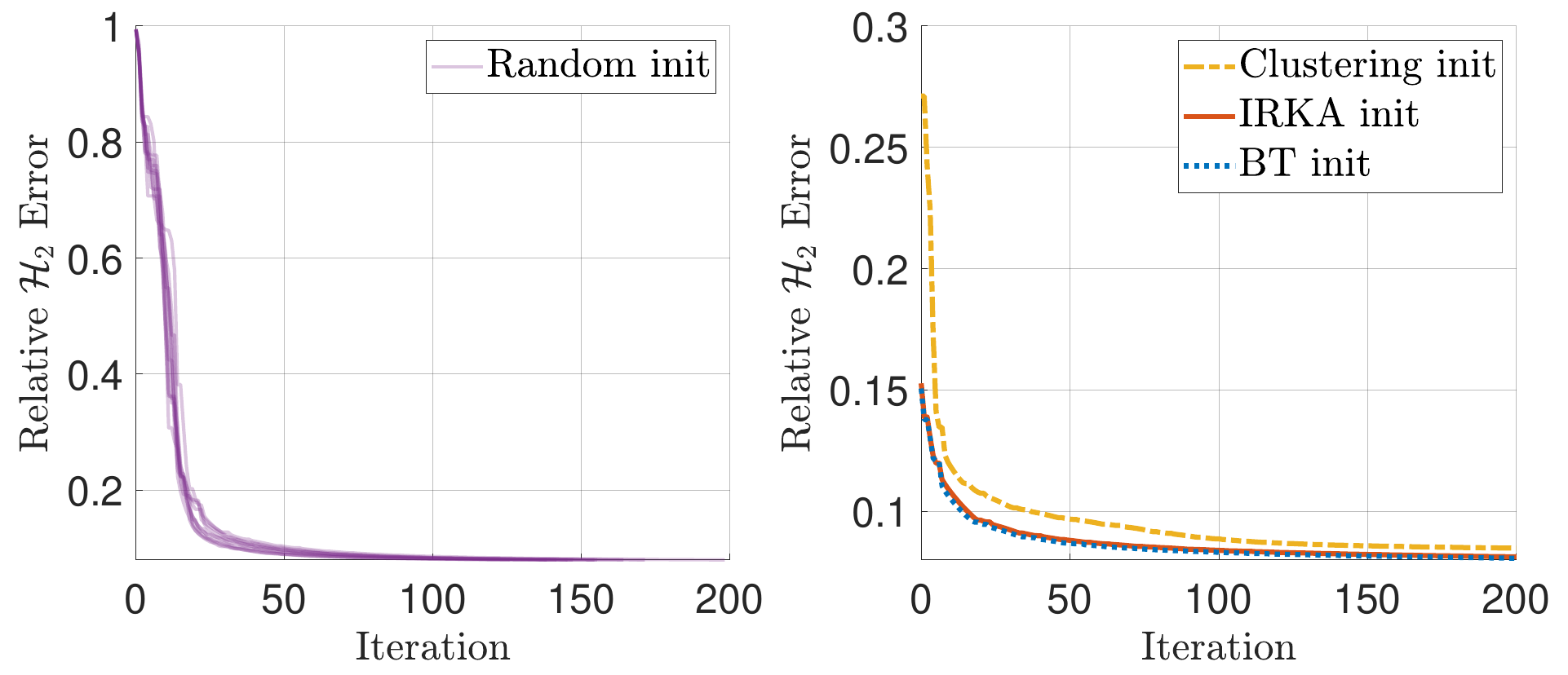}
		\caption{Convergence behavior of OPMR for different initializations. Left: 10 random full-rank initializations. Right: clustering-, IRKA-, and BT-based initializations.}
		\label{fig:converg}
	\end{figure}

\section{Conclusions}
\label{sec:conclusion}

This paper has presented a novel structure-preserving model reduction framework that combines Petrov-Galerkin
projection with nonconvex optimization.  By selecting
an appropriate structure matrix, the framework
preserves various properties of linear systems,
including asymptotic stability, bounded realness,
conic sector boundedness, and dissipativity, under
the corresponding matrix inequality conditions.
The model reduction is formulated as an $\mathcal{H}_2$ optimization problem on a generalized Stiefel manifold. We derived the Euclidean and Riemannian gradients of the objective function and developed an efficient generalized Cayley algorithm to seek a (locally) optimal projection.
The proposed framework focuses on continuous-time
systems, but extending it to discrete-time
systems using appropriate discrete-time Lyapunov
and Sylvester equations and matrix inequalities
is a possible direction for future work.


\section*{Appendix}
\renewcommand{\thesubsection}{\Alph{subsection}}

\subsection{Proof of Theorem \ref{thm:preservation}}
\label{ap:thm:preservation}

Let $\hat K=V^\top K V\succ0$.
If $K$ satisfies \eqref{eq:K_stability}, then
$
	\hat A^\top\hat K+\hat K\hat A
	=V^\top(A^\top K+KA)V\prec0.
$
Therefore, the reduced-order model
$\mathbf{\hat{\Sigma}}$ is asymptotically stable,
which proves Statement 1.

The conditions for passivity, finite-gain $L_2$
stability, and conic sector boundedness in
Table~\ref{tab:properties} are special cases
of the $(Q,S,R)$-dissipativity condition
\eqref{eq:K_QSR}, with particular choices
of $Q$, $S$, and $R$.
We therefore first establish the preservation
of $(Q,S,R)$-dissipativity.
Using $\hat K=V^\top K V$, $\hat C=CV$,
and $\hat D=D$, we obtain
\begin{align*}
	&\hat A^\top\hat K+\hat K\hat A
	-\hat C^\top Q\hat C =V^\top(A^\top K+KA-C^\top QC)V,\\
	&\hat K\hat B-\hat C^\top S-\hat C^\top QD =V^\top(KB-C^\top S-C^\top QD).
\end{align*}
Let $\Gamma$ denote the matrix on the
left-hand side of \eqref{eq:K_QSR}.
Then, the corresponding
matrix for the reduced-order model satisfies
\begin{align*}
	&\begin{bmatrix}
		\hat A^\top\hat K+\hat K\hat A
		-\hat C^\top Q\hat C
		&
		\hat K\hat B-\hat C^\top S
		-\hat C^\top QD\\
		\star
		&
		-D^\top QD-D^\top S-S^\top D-R
	\end{bmatrix}\\
	&=
	\begin{bmatrix}
		V^\top & 0\\
		0 & I_p
	\end{bmatrix}
	\Gamma
	\begin{bmatrix}
		V & 0\\
		0 & I_p
	\end{bmatrix} \preceq0,
\end{align*}
where the inequality follows from
$\Gamma\preceq0$.

Since $\hat K\succ0$, the reduced-order model
satisfies the corresponding dissipativity
inequality with the same supply-rate
matrices $Q$, $S$, and $R$.
Therefore, passivity, finite-gain $L_2$
stability with the same gain bound $\gamma$,
conic sector boundedness in $[a,b]$,
and $(Q,S,R)$-dissipativity are preserved,
respectively, under the conditions
in Statements 2-5. The property of asymptotic stability
follows directly from the additional
condition \eqref{eq:K_stability}.
This completes the proof.

\subsection{Proof of Lemma \ref{lem:gradJ}}
\label{ap:lem:gradJ}
Observe that
$
V^\top X = (V^\top X V) V^\dagger.
$
Differentiating both sides gives
$
\diff V^\top X =  \diff((V^\top X V) V^\dagger)
=                 (\diff V^\top X V + V^\top X \diff V) V^\dagger + (V^\top X V)  \diff (V^\dagger),
$
from which, the differential $ \diff (V^\dagger)$ is solved as
\begin{equation}\label{eq:differential0}
	\diff(V^\dagger) = (V^\top X V)^{-1}  \diff V^\top X (I - V V^\dagger)  - V^\dagger \diff V V^\dagger.
\end{equation}

To compute the gradient $\nabla J(V)$,  we write the derivative  $J'(V) \diff{V}$ for an arbitrary perturbation $\diff{V} \in  \mathbb R^{n \times r }$ in a gradient form in terms of trace as $
J'(V)  \diff{V}   = \tr \left( \nabla J(V)^\top \diff{V}  \right).
$
Moreover, if we define the objective function $J(V) := F(\hat{A}, \hat{B}, \hat{C})$ with
$
\hat{A} = V^\dagger A V, \ \hat{B} = V^\dagger B, \ \text{and} \ \hat{C} = C V,
$
we have
\begin{align}
	\label{eq:dV-dAdBdC}
	& \tr \left( \nabla J(V)^\top \diff{V}  \right)   \nonumber                                   \\
	= & \tr \left( \nabla F(\hat{A})^\top \diff{\hat{A}} +  \nabla F(\hat{B})^\top \diff{\hat{B}} +
	\nabla F(\hat{C})^\top \diff{\hat{C}}\right),
\end{align}
where $\nabla F(\hat{A})$, $\nabla F(\hat{B})$, and $\nabla F(\hat{C})$ are gradients of the objective function $F(\hat{A}, \hat{B}, \hat{C})$. Following \cite{vanDooren2008H2}, the gradients can be determined as
\begin{subequations}
	\begin{align}
		\nabla F(\hat{A}) & = 2(   \calQ_{12}^\top\calP_{12}  +  \calQ_{22}\calP_{22}) = 2 \Sigma_{22}^\top, \\
		\nabla F(\hat{B}) & = 2 (\calQ_{12}^\top + \calQ_{22} V^\dagger) B,                                  \\
		\nabla F(\hat{C}) & = 2 C (V \calP_{22} -  \calP_{12}).
	\end{align}
\end{subequations}
As a result, we derive the following equations by using \eqref{eq:differential0}:
\begin{subequations}
	\begin{align}
		\label{eq:trace_diffMB1}
		& \tr (\nabla F (\hat{A})^\top \diff \hat{A}(V))                                                             \\
		& =
		2\tr \left[ (V^\top X V)^{-1} \F_A(V)^\top (I - (V V^\dagger)^\top) X  \diff V   \right. \nonumber            \\
		& \quad \left. - V^\dagger  \F_A(V) V^\dagger \diff V   + \Sigma_{22} V^\dagger A \diff V  \right],\nonumber
		\\
		\label{trace_MdiffB}
		& \tr \left(\nabla F (\hat{B})^\top \diff \hat{B}(V) \right)
		\\
		& =  2\tr \left[ (V^\top X V)^{-1} \F_B(V)(I - (VV^\dagger)^\top) X \diff V  \right. \nonumber
		\\ & \left. \quad - V^\dagger  \F_B(V)^\top  V^\dagger  \diff V
		\right],
		\nonumber
		\\
		\label{eq:trace_diffMB2}
		& \tr(\nabla F (\hat{C})^\top \diff \hat{C}(V))
		= 2 \tr(\F_C(V)^\top \diff{V}),
	\end{align}
\end{subequations}
with $\F_A(V)$, $\F_B(V)$, and $\F_C(V)$ defined in \eqref{eq:FAFBFC}.
Finally, we substitute  the three equations \eqref{trace_MdiffB}, \eqref{eq:trace_diffMB1}, and \eqref{eq:trace_diffMB2} into \eqref{eq:dV-dAdBdC} to obtain the gradient in \eqref{eq:nablaJ}.

\subsection{Proof of Theorem \ref{thm:gradflow}}
\label{ap:thm:gradflow}
The basis invariance in \eqref{eq:basis_invariance} gives $J(VT)=J(V)$ for every nonsingular $T$. For an arbitrary matrix $H\in\mathbb{R}^{r\times r}$, differentiating $J(V(I_r+tH))$ at $t=0$ yields
$
	0=\tr\!\left(\nabla J(V)^\top VH\right)
	=\tr\!\left(H\nabla J(V)^\top V\right).
$
Since $H$ is arbitrary, $V^\top\nabla J(V)=0$, which proves \eqref{eq:vertical_orthogonality}.
Moreover, for every $Z\in T_V\mathcal M$,
\[
	g_V\!\left(M^{-1}\nabla J(V),Z\right)
	=\tr\!\left(\nabla J(V)^\top Z\right).
\]
Thus, $M^{-1}\nabla J(V)$ is the ambient gradient induced by \eqref{eq:M_metric}. Furthermore,
\begin{align*}
	V^\top M\big(M^{-1}\nabla J(V)\big)
	&=V^\top\nabla J(V)=0,\\
	\big(M^{-1}\nabla J(V)\big)^\top MV
	&=\nabla J(V)^\top V=0.
\end{align*}
Therefore, this ambient gradient belongs to $T_V\mathcal M$ and is the Riemannian gradient, which proves the gradient assertion.

Along \eqref{eq:gradflow}, it follows that
$
	V(t)^\top M\dot V(t)=-V(t)^\top\nabla J(V(t))=0.
$
Consequently,
\[
	\frac{\diff}{\diff t}(V(t)^\top M V(t))
	=\dot V(t)^\top M V(t)+V(t)^\top M\dot V(t)=0,
\]
and therefore $V(t)^\top M V(t)=I_r$ for all $t\geq0$. Hence, $V(t)\in\mathcal M$ for all $t\geq0$, which proves the invariance assertion.

The derivative of the objective along the flow is
\begin{align*}
	\frac{\diff}{\diff t}J(V(t))
	&=\tr\!\left(\nabla J(V(t))^\top\dot V(t)\right)\\
	&=-\tr\!\left(\nabla J(V(t))^\top M^{-1}\nabla J(V(t))\right)\leq0,
\end{align*}
which proves \eqref{eq:nonincreasing}.

Finally, $\nabla J(V)=0$ clearly implies $W_M(V)=0$. Conversely, using $V^\top M V=I_r$ and $\nabla J(V)^\top V=0$, we obtain
\[
	W_M(V)V
	=M^{-1}\nabla J(V).
\]
Thus, $W_M(V)=0$ implies $\nabla J(V)=0$, which proves \eqref{eq:JV0}.

\subsection{Proof of Theorem \ref{thm:alg2}}
\label{ap:thm:alg2}

Let
$
	Y=M^{1/2}V,
	\ 
	\widetilde G=M^{-1/2}\nabla J(V).
$
Then $Y^\top Y=I_r$, $Y^\top\widetilde G=V^\top\nabla J(V)=0$, and
$\widetilde G^\top\widetilde G=\nabla J(V)^\top M^{-1}\nabla J(V)$. Consider the standard Cayley curve in transformed coordinates,
\begin{equation}\label{eq:Cayley1}
	\widetilde{\mathcal V}(\tau)
	=\left(I+\frac{\tau}{2}\widetilde W\right)^{-1}
	\left(I-\frac{\tau}{2}\widetilde W\right)Y,
\end{equation}
where
$
	\widetilde W=\widetilde GY^\top-Y\widetilde G^\top
$
is skew-symmetric. Hence,
$
	\widetilde{\mathcal V}(\tau)^\top
	\widetilde{\mathcal V}(\tau)=Y^\top Y=I_r.
$
Equivalently, the mapped curve $\mathcal V(\tau)=M^{-1/2}\widetilde{\mathcal V}(\tau)$ preserves $\mathcal V(\tau)^\top M\mathcal V(\tau)=I_r$. In the original coordinates, its generator is $W_M(V)$, which satisfies
\[
	W_M(V)^\top M+M W_M(V)=0.
\]%
Using the matrix inversion lemma and $Y^\top\widetilde G=0$, the Cayley curve \eqref{eq:Cayley1} admits the low-rank representation
\begin{align*}
	\widetilde{\mathcal V}(\tau)
	=Y-\tau\left(I+\frac{\tau}{2}Y\widetilde G^\top\right)
	\widetilde G
	\left(I_r+\frac{\tau^2}{4}\widetilde G^\top\widetilde G\right)^{-1}.
\end{align*}
Multiplying this identity from the left by $M^{-1/2}$ yields
\begin{align*}
	\mathcal V(\tau)
	={}&V-\tau\left(I+\frac{\tau}{2}V\nabla J(V)^\top\right)
	M^{-1}\nabla J(V)\\
	&\quad\times
	\left(I_r+\frac{\tau^2}{4}\nabla J(V)^\top M^{-1}\nabla J(V)\right)^{-1},
\end{align*}
which is \eqref{eq:Vtau}. This proves the first statement of Theorem~\ref{thm:alg2}.

Differentiating the curve at $\tau=0$ gives 
$\mathcal V_k'(0)=-M^{-1}\nabla J(V_k).$
Therefore,
\[
	\phi_k'(0)
	=\left\langle \nabla J(V_k),\mathcal V_k'(0)\right\rangle
	=-\tr\!\left(\nabla J(V_k)^\top M^{-1}\nabla J(V_k)\right).
\]
Thus, the curve gives a strict descent direction with $\nabla J(V_k)\neq0$. Substituting into the Armijo condition \eqref{eq:Armijo2} gives \eqref{eq:convergence}.

The sequence $\{J(V_k)\}$ is monotonically
nonincreasing and convergent since $J(V)\geq0$.
The Cayley search curve is smooth, remains on
$\mathcal M$, and satisfies
$\mathcal V_k'(0)=-\operatorname{grad}J(V_k)$.
Since $J$ is smooth and $\mathcal M$ is compact,
the Armijo backtracking scheme satisfies the
standard convergence conditions for
gradient-related line-search methods.
Consequently, every accumulation point of
$\{V_k\}$ is a critical point of $J$.
By \eqref{eq:Riemannian_gradient} and $M\succ0$,
this is equivalent to $\nabla J(V_\star)=0$.

	\bibliographystyle{IEEEtran}
	\bibliography{structure}

\begin{thebibliography}{10}
\providecommand{\url}[1]{#1}
\csname url@samestyle\endcsname
\providecommand{\newblock}{\relax}
\providecommand{\bibinfo}[2]{#2}
\providecommand{\BIBentrySTDinterwordspacing}{\spaceskip=0pt\relax}
\providecommand{\BIBentryALTinterwordstretchfactor}{4}
\providecommand{\BIBentryALTinterwordspacing}{\spaceskip=\fontdimen2\font plus
\BIBentryALTinterwordstretchfactor\fontdimen3\font minus
  \fontdimen4\font\relax}
\providecommand{\BIBforeignlanguage}[2]{{%
\expandafter\ifx\csname l@#1\endcsname\relax
\typeout{** WARNING: IEEEtran.bst: No hyphenation pattern has been}%
\typeout{** loaded for the language `#1'. Using the pattern for}%
\typeout{** the default language instead.}%
\else
\language=\csname l@#1\endcsname
\fi
#2}}
\providecommand{\BIBdecl}{\relax}
\BIBdecl

\bibitem{obinata2012model}
G.~Obinata and B.~D.~O. Anderson, \emph{Model Reduction for Control System
  Design}.\hskip 1em plus 0.5em minus 0.4em\relax Springer Science \& Business
  Media, 2012.

\bibitem{Gugercin2004Survey}
S.~Gugercin and A.~C. Antoulas, ``A survey of model reduction by balanced
  truncation and some new results,'' \emph{International Journal of Control},
  vol.~77, no.~8, pp. 748--766, 2004.

\bibitem{cheng2021review}
X.~Cheng and J.~M.~A. Scherpen, ``Model reduction methods for complex network
  systems,'' \emph{Annual Review of Control, Robotics, and Autonomous Systems},
  vol.~4, pp. 425--453, 2021.

\bibitem{moore1981principal}
B.~C. Moore, ``Principal component analysis in linear systems: Controllability,
  observability, and model reduction,'' \emph{IEEE Transactions on Automatic
  Control}, vol.~26, no.~1, pp. 17--32, 1981.

\bibitem{polyuga2012effort}
R.~V. Polyuga and A.~J. van~der Schaft, ``Effort-and flow-constraint reduction
  methods for structure preserving model reduction of port-{Hamiltonian}
  systems,'' \emph{Systems \& Control Letters}, vol.~61, no.~3, pp. 412--421,
  2012.

\bibitem{guiver2013error}
C.~Guiver and M.~R. Opmeer, ``Error bounds in the gap metric for dissipative
  balanced approximations,'' \emph{Linear Algebra and Its Applications}, vol.
  439, no.~12, pp. 3659--3698, 2013.

\bibitem{kawano2018structure}
Y.~Kawano and J.~M.~A. Scherpen, ``Structure preserving truncation of nonlinear
  port-{Hamiltonian} systems,'' \emph{IEEE Transactions on Automatic Control},
  vol.~63, no.~12, pp. 4286--4293, 2018.

\bibitem{salehi2021passivity}
Z.~Salehi, P.~Karimaghaee, and M.-H. Khooban, ``A new passivity preserving
  model order reduction method: conic positive real balanced truncation
  method,'' \emph{IEEE Transactions on Systems, Man, and Cybernetics: Systems},
  2021.

\bibitem{XiaodongBT2017}
X.~Cheng, J.~M. Scherpen, and B.~Besselink, ``Balanced truncation of networked
  linear passive systems,'' \emph{Automatica}, vol. 104, pp. 17--25, 2019.

\bibitem{antoulas2005passive}
A.~C. Antoulas, ``A new result on passivity preserving model reduction,''
  \emph{Systems \& control letters}, vol.~54, no.~4, pp. 361--374, 2005.

\bibitem{sorensen2005passivity}
D.~C. Sorensen, ``Passivity preserving model reduction via interpolation of
  spectral zeros,'' \emph{Systems \& Control Letters}, vol.~54, no.~4, pp.
  347--360, 2005.

\bibitem{gugercin2012pH}
S.~Gugercin, R.~V. Polyuga, C.~Beattie, and A.~J. Van Der~Schaft,
  ``Structure-preserving tangential interpolation for model reduction of
  port-{Hamiltonian} systems,'' \emph{Automatica}, vol.~48, no.~9, pp.
  1963--1974, 2012.

\bibitem{polyuga2010pH}
R.~V. Polyuga and A.~J. Van~der Schaft, ``Structure preserving model reduction
  of port-{Hamiltonian} systems by moment matching at infinity,''
  \emph{Automatica}, vol.~46, no.~4, pp. 665--672, 2010.

\bibitem{wolf2010passivity}
T.~Wolf, B.~Lohmann, R.~Eid, and P.~Kotyczka, ``Passivity and structure
  preserving order reduction of linear port-{Hamiltonian} systems using krylov
  subspaces,'' \emph{European Journal of Control}, vol.~16, no.~4, pp.
  401--406, 2010.

\bibitem{Ionescu2013}
T.~C. Ionescu and A.~Astolfi, ``{Families of moment matching based, structure
  preserving approximations for linear port-{Hamiltonian} systems},''
  \emph{Automatica}, vol.~49, no.~8, pp. 2424--2434, 2013.

\bibitem{breiten2022passivity}
T.~Breiten and B.~Unger, ``Passivity preserving model reduction via spectral
  factorization,'' \emph{Automatica}, vol. 142, p. 110368, 2022.

\bibitem{afkham2019structure}
B.~M. Afkham and J.~S. Hesthaven, ``Structure-preserving model-reduction of
  dissipative {Hamiltonian} systems,'' \emph{Journal of Scientific Computing},
  vol.~81, no.~1, pp. 3--21, 2019.

\bibitem{Yu1999approximate}
W.-Y. Yan and J.~Lam, ``An approximate approach to ${H}^2$ optimal model
  reduction,'' \emph{IEEE Transactions on Automatic Control}, vol.~44, no.~7,
  pp. 1341--1358, 1999.

\bibitem{vanDooren2008H2}
P.~Van~Dooren, K.~A. Gallivan, and P.-A. Absil, ``{$H_2$}-optimal model
  reduction of mimo systems,'' \emph{Applied Mathematics Letters}, vol.~21,
  no.~12, pp. 1267--1273, 2008.

\bibitem{beattie2007krylov}
C.~A. Beattie and S.~Gugercin, ``Krylov-based minimization for optimal {$H_2$}
  model reduction,'' in \emph{Proc. 46th IEEE Conference on Decision and
  Control}.\hskip 1em plus 0.5em minus 0.4em\relax IEEE, 2007, pp. 4385--4390.

\bibitem{beattie2009trust}
------, ``A trust region method for optimal {$H_2$} model reduction,'' in
  \emph{Proc. 48h IEEE Conference on Decision and Control (CDC) held jointly
  with the 28th Chinese Control Conference}.\hskip 1em plus 0.5em minus
  0.4em\relax IEEE, 2009, pp. 5370--5375.

\bibitem{gugercin2008H2}
S.~Gugercin, A.~C. Antoulas, and C.~Beattie, ``{$H_2$} model reduction for
  large-scale linear dynamical systems,'' \emph{SIAM Journal on Matrix Analysis
  and Applications}, vol.~30, no.~2, pp. 609--638, 2008.

\bibitem{sato2017structure}
K.~Sato and H.~Sato, ``Structure-preserving ${H}^2$ optimal model reduction
  based on the {Riemannian} trust-region method,'' \emph{IEEE Transactions on
  Automatic Control}, vol.~63, no.~2, pp. 505--512, 2017.

\bibitem{sato2018pHsys}
K.~Sato, ``Riemannian optimal model reduction of linear port-{Hamiltonian}
  systems,'' \emph{Automatica}, vol.~93, pp. 428--434, 2018.

\bibitem{sato2021book}
H.~Sato, \emph{Riemannian Optimization and Its Applications}.\hskip 1em plus
  0.5em minus 0.4em\relax Springer, 2021.

\bibitem{Jiang2019model}
Y.-L. Jiang and K.-L. Xu, ``Model order reduction of port-{Hamiltonian} systems
  by riemannian modified fletcher--reeves scheme,'' \emph{IEEE Transactions on
  Circuits and Systems II: Express Briefs}, vol.~66, no.~11, pp. 1825--1829,
  2019.

\bibitem{moser2020new}
T.~Moser and B.~Lohmann, ``A new {Riemannian} framework for efficient
  {$H_2$}-optimal model reduction of port-{Hamiltonian} systems,'' in
  \emph{Proc. 59th IEEE Conference on Decision and Control (CDC)}.\hskip 1em
  plus 0.5em minus 0.4em\relax IEEE, 2020, pp. 5043--5049.

\bibitem{schwerdtner2023optimization}
P.~Schwerdtner, T.~Moser, V.~Mehrmann, and M.~Voigt, ``Optimization-based model
  order reduction of port-hamiltonian descriptor systems,'' \emph{Systems \&
  Control Letters}, vol. 182, p. 105655, 2023.

\bibitem{willems2007dissipative}
J.~C. Willems, ``Dissipative dynamical systems,'' \emph{European Journal of
  Control}, vol.~13, no. 2-3, pp. 134--151, 2007.

\bibitem{xia2016passivity}
M.~Xia, P.~J. Antsaklis, V.~Gupta, and F.~Zhu, ``Passivity and dissipativity
  analysis of a system and its approximation,'' \emph{IEEE Transactions on
  Automatic Control}, vol.~62, no.~2, pp. 620--635, 2016.

\bibitem{hauschild2019pH}
S.-A. Hauschild, N.~Marheineke, and V.~Mehrmann, ``Model reduction techniques
  for linear constant coefficient port-{Hamiltonian} differential-algebraic
  systems,'' \emph{arXiv preprint arXiv:1901.10242}, 2019.

\bibitem{sun2006optimization}
W.~Sun and Y.-X. Yuan, \emph{Optimization Theory and Methods: Nonlinear
  Programming}.\hskip 1em plus 0.5em minus 0.4em\relax Springer Science \&
  Business Media, 2006, vol.~1.

\bibitem{Absil2009optimization}
P.-A. Absil, R.~Mahony, and R.~Sepulchre, \emph{Optimization Algorithms on
  Matrix Manifolds}.\hskip 1em plus 0.5em minus 0.4em\relax Princeton
  University Press, 2009.

\end{thebibliography}
	
		%
	\end{document}